\documentclass{article}
\usepackage{amsmath,latexsym}
\usepackage{enumerate}
\usepackage{amsthm} 

\newcommand\RR{\rm {I \! R}}
\newcommand\ZZ{\mathbf{Z}}
\newcommand\NN{\rm {I \! N}}
\newcommand\FF{{\cal F}}

\DeclareMathOperator\sP{P}   
\newcommand{\rP}{\mathrm{P}} 
\DeclareMathOperator\sE{E}   
\newcommand{\rE}{\mathrm{E}}

\newcommand{\eps}{\varepsilon}

\numberwithin{equation}{section}
\theoremstyle{plain}
\newtheorem{theorem}{Theorem}[section]

\newtheorem{corollary}{Corollary}[section]
\newtheorem{example}{Example}[section]
\newtheorem{lemma}{Lemma}[section]

\newcommand\address{Address correspondence to P. C. Allaart, Department of Mathematics, University of North Texas,
1155 Union Circle \#311430, Denton, TX 76203-5017, USA; E-mail: allaart@unt.edu}
\newcommand\thankyou{Supported in part by Japanese GCOE Program G08: ``Fostering Top Leaders in Mathematics --- Broadening the Core and Exploring New Ground".} 

\title{A general ``bang-bang" principle for predicting the maximum of a random walk\footnote{\thankyou}}
\author{Pieter C. Allaart \footnote{\address}}

\begin{document}

\maketitle

\begin{abstract}
Let $(B_t)_{0\leq t\leq T}$ be either a Bernoulli random walk or a Brownian motion with drift, and let $M_t:=\max\{B_s: 0\leq s\leq t\}$, $0\leq t\leq T$. This paper solves the general optimal prediction problem
\begin{equation*}
\sup_{0\leq\tau\leq T}\sE[f(M_T-B_\tau)],
\end{equation*}
where the supremum is over all stopping times $\tau$ adapted to the natural filtration of $(B_t)$, and $f$ is a nonincreasing convex function. The optimal stopping time $\tau^*$ is shown to be of ``bang-bang" type: $\tau^*\equiv 0$ if the drift of the underlying process $(B_t)$ is negative, and $\tau^*\equiv T$ is the drift is positive. This result generalizes recent findings by S.~Yam, S.~Yung and W.~Zhou [{\em J. Appl. Probab.} {\bf 46} (2009), 651--668] and J.~Du~Toit and G.~Peskir [{\em Ann. Appl. Probab.} {\bf 19} (2009), 983--1014], and provides additional mathematical justification for the dictum in finance that one should sell bad stocks immediately, but keep good ones as long as possible.



\end{abstract}

\section{Introduction and main results}

A number of recent papers (e.g. \cite{DuToit,SXZ,YYZ}) have discussed the problem of stopping a random walk, or a Brownian motion, ``as close as possible" to its ultimate maximum. An important motivation in these papers was the financial problem of selling a stock at a price ``close" to the highest price over a given finite time interval, when the stock price follows a discrete binomial model (in \cite{YYZ}) or a geometric Brownian motion (in \cite{DuToit,SXZ}). In these three papers, ``closeness" was measured by the ratio of the stopped price to the ultimate maximum price, and the striking result was that the optimal strategy is of ``bang-bang" type, meaning that it is either optimal to stop at time zero, or to stop at the time horizon, depending on the quality of the stock. These results, as pointed out by the papers' authors, reinforce the widely held financial view that one should sell bad stocks quickly, but keep good ones as long as possible.

The purpose of the present paper is to provide an important generalization of the results in \cite{DuToit,YYZ}. Rather than considering price ratios, we take as the basic process either a ``flat" Bernoulli random walk or a Brownian motion with drift, and measure closeness by a general nonincreasing, convex function $f$ of the positive distance from the stopped value of the process to its eventual maximum. For this more general problem we obtain the same result, namely that it is either optimal to stop at time zero or at the time horizon, depending on the drift of the underlying process. For the specific function $f(x)=e^{-\sigma x}$, where $\sigma>0$, our results reduce to those of \cite{DuToit} and \cite{YYZ}. The proofs involve only a minimum of technicalities, and bring to the foreground the essential feature hidden within the arguments in the aforementioned papers, namely convexity of the function $f$.

The remainder of this section is devoted to a precise formulation of the problem and statements of the main results, which are nontechnical in nature. First, let $\{S_n\}_{n=0,1,\dots}$ be a Bernoulli random walk with  parameter $p\in(0,1)$. That is, $S_0\equiv 0$, and for $n\geq 1$, $S_n=X_1+\dots+X_n$, where $X_1,X_2,\dots$ are independent, identically distributed random variables with $\rP(X_1=1)=p$, and $\rP(X_1=-1)=q:=1-p$. Let a finite time horizon $N\in\NN$ be given.
Let $f:\{0,1,\dots,N\}\to\RR$ be nonincreasing, and consider the optimal stopping problem
\begin{equation}
\sup_{0\leq\tau\leq N}\rE[f(M_N-S_\tau)],
\label{eq:objective}
\end{equation}
where $M_N:=\max\{S_0,S_1,\dots,S_N\}$, and the supremum is over the set of all stopping times $\tau\leq N$ adapted to the natural filtration $\{\FF_k\}_{0\leq k\leq N}$ of the process $\{S_k\}_{0\leq k\leq N}$.

As a concrete example, taking $f(0)=1$ and $f(k)=0$ for $k\geq 1$ turns the expectation in \eqref{eq:objective} into the probability $\rP(S_\tau=M_N)$, so that \eqref{eq:objective} becomes a ``best-choice" or ``secretary" problem for the random walk, where the goal is to maximize the probability of stopping at the ultimate maximum of the walk; see \cite{Hlynka}, where this problem is solved in a somewhat more general setting for the case $p=1/2$. Yam et al. \cite{YYZ} solved the problem for arbitrary $p$, and showed the (unique) optimal rule to be $\tau\equiv 0$ when $p<1/2$, and $\tau\equiv n$ when $p>1/2$. When $p=1/2$, it is optimal to stop at time $0$, or at time $N$, or at any time at which the walk is at its running maximum. (A similar problem, where the objective is to stop a Brownian motion within a distance $\eps>0$ from its ultimate maximum, was considered in \cite{Pedersen}.)

Yam et al. \cite{YYZ} also treated the case $f(k)=d^k$, where $0<d<1$ is a constant. They showed that for this quite different objective function, the optimal rule is nonetheless exactly the same as for the problem of maximizing the probability of stopping at the maximum.

This leads one to believe that there must be some general principle at work.
Notice that in each of the above examples, $f$ is in fact convex. The first aim of this note is to show that the optimal rule is of the above simple form for any nonincreasing convex objective function $f$, thereby generalizing the results of \cite{YYZ}. Recall that a function $f:\{0,1,\dots,N\}\to\RR$ is {\em convex} if $f(k)-2f(k+1)+f(k+2)\geq 0$ for all $k$ with $0\leq k\leq N-2$, and is {\em strictly convex} if the inequality is strict for all such $k$.

\begin{theorem} \label{thm:optimal-rules}
Let $f:\{0,1,\dots,N\}\to\RR$ be nonincreasing and convex, and consider the optimal stopping problem \eqref{eq:objective}.

(i) If $p\leq 1/2$, the rule $\tau\equiv 0$ is optimal.

(ii) If $p\geq 1/2$, the rule $\tau\equiv N$ is optimal.

(iii) If $p=1/2$, any stopping time $\tau$ satisfying $S_\tau=M_\tau$ or $\tau=N$ almost surely is optimal.
\end{theorem}

Thus, in the words of Du Toit and Peskir \cite{DuToit}, the optimal strategy $\tau^*$ is of ``bang-bang" type: $\tau^*\equiv 0$ if $p<1/2$, and $\tau^*\equiv N$ if $p>1/2$.

Convexity of $f$ is essential, as the following example shows.

\begin{example}
{\rm
Let $f(0)=f(1)=1$, and $f(k)=0$ for $k\geq 2$. Thus, there are two possible outcomes, ``winning" and ``losing", and we win if we stop with one of the two highest values of the walk. Let $N=2$. It is easy to see that, regardless of $p$, the rule $\tau\equiv 1$ gives a winning probability of $1$. On the other hand, the winning probability for the rule $\tau\equiv 0$ is $1-p^2$, and that for the rule $\tau\equiv 2$ is $1-q^2$.
}
\end{example}

One might ask when the optimal rules in Theorem \ref{thm:optimal-rules} are unique. The next theorem gives simple sufficient conditions to this effect.

\begin{theorem} \label{thm:unique}
Let $f$ be as in Theorem \ref{thm:optimal-rules}.

(i) If $p<1/2$ and $f$ is nonconstant, then the rule $\tau\equiv 0$ is the unique optimal rule.

(ii) If $p>1/2$ and $f$ is strictly decreasing, then the rule $\tau\equiv N$ is the unique optimal rule.

(iii) If $p=1/2$ and $f$ is strictly convex, then the only optimal rules are those that satisfy $S_\tau=M_\tau$ or $\tau=N$ almost surely.
\end{theorem}

It is left to the interested reader to verify that the above conditions can not be substantially weakened.

\bigskip
Next, let $B:=(B_t)_{t\geq 0}$ be a standard Brownian motion, and $\lambda$ a real parameter. Then the process ${(B_t^\lambda)}_{t\geq 0}$ defined by $B_t^\lambda:=B_t+\lambda t$ is a Brownian motion with drift $\lambda$. Let $M_t^\lambda:=\max\{B_s^\lambda: 0\leq s\leq t\}$. Once again we are interested in finding a stopping time $\tau$ (with respect to the natural filtration ${(\FF_t^B)}_{t\geq 0}$ of $B$) that will attain the maximum in
\begin{equation}
\sup_{0\leq\tau\leq T}\rE[f(M_T^\lambda-B_\tau^\lambda)],
\label{eq:BM-stopping-problem}
\end{equation}
where $f:[0,\infty)\to\RR$ is a given (reward) function, and $T>0$ a fixed time horizon. Note that in particular, the choice $f(x)=e^{-\sigma x}$ for a constant $\sigma>0$ yields the problem studied in \cite{SXZ} and in Section 4 of \cite{DuToit}.

Since Brownian motion is the scaling limit of Bernoulli random walk, one might expect the result to be the same as in Theorem \ref{thm:optimal-rules}. This is indeed the case, except that the conditions for uniqueness of the optimal rules are weaker.

\begin{theorem} \label{thm:BM-optimal-rules}
Let $f:[0,\infty)\to\RR$ be non-constant, nonincreasing and convex, and consider the optimal stopping problem \eqref{eq:BM-stopping-problem}.

(i) If $\lambda<0$, the rule $\tau\equiv 0$ is the unique optimal rule.

(ii) If $\lambda>0$, the rule $\tau\equiv T$ is the unique optimal rule.

(iii) If $\lambda=0$, any stopping time $\tau$ satisfying 
\begin{equation}
\rP(B_\tau^\lambda=M_\tau^\lambda\ \mbox{or}\ \tau=T)=1
\label{eq:indifference}
\end{equation}
is optimal. (In particular, the rules $\tau\equiv 0$ and $\tau\equiv T$ are optimal.) If, furthermore, $f$ is not linear, then all optimal rules satisfy \eqref{eq:indifference}.
\end{theorem}

Note that if $f$ is constant, or if $f$ is linear and $\lambda=0$, then any stopping time is optimal in view of the optional sampling theorem. Thus, the uniqueness conditions in Theorem \ref{thm:BM-optimal-rules} are the best possible. Essentially, the conditions for uniqueness of the optimal rules are weaker than in the discrete case because the increments of Brownian motion can be arbitrarily large, whereas the increments of Bernoulli random walk are bounded.

Finally, we note that by putting $\tilde{f}:=-f$, problems \eqref{eq:objective} and \eqref{eq:BM-stopping-problem} may be formulated equivalently as penalty-minimization problems. For instance, \eqref{eq:BM-stopping-problem} can be represented alternatively in the form
\begin{equation}
\inf_{0\leq\tau\leq T}\sE[\tilde{f}(M_T^\lambda-B_\tau^\lambda)],
\label{eq:minimization-problem}
\end{equation}
where $\tilde{f}:[0,\infty)\to\RR$ is nondecreasing and concave. Thus, the above results apply to a variety of natural penalty functions, including $\tilde{f}(x)=x^\alpha$ where $0<\alpha<1$,  $\tilde{f}(x)=\log(1+x)$, etc. However, without concavity of $\tilde{f}$ the optimal rules are generally of a more intricate form: see, for instance, the solution in \cite[Section 3]{DuToit} of \eqref{eq:minimization-problem} for the function $\tilde{f}(x)=e^{\sigma x}$, where $\sigma>0$.

\bigskip
Theorems \ref{thm:optimal-rules} and \ref{thm:unique} are proved in Section \ref{sec:random-walk}, and Theorem \ref{thm:BM-optimal-rules} is proved in Section \ref{sec:Brownian}. Many of the ideas of the proofs are adapted from \cite{DuToit} and \cite{YYZ}, and some details, in as far as they can be found in these papers, are therefore omitted here. The novel contributions of the present article are the explicit use of the convexity of $f$ (see Lemmas \ref{lem:key-inequality} and \ref{lem:BM-key} below), and the investigation of uniqueness of the optimal stopping times, which requires some finesse in the case of general $f$.

\section{The maximum of Bernoulli random walk} \label{sec:random-walk}

This section is devoted to the proofs of Theorems \ref{thm:optimal-rules} and \ref{thm:unique}. It will be useful to consider an infinite family of random walks, defined on the same probability space. The following construction is standard. Let $U_1,U_2\dots$ be independent random variables, uniformly distributed on the interval $[0,1]$. For $k\in\NN$ and $p\in(0,1)$, define
\begin{equation*}
X_k^p:=\begin{cases}
1, & \mbox{if $U_k\leq p$}\\
-1, & \mbox{if $U_k>p$}.
\end{cases}
\end{equation*}
Define $S_0^p\equiv 0$, and $S_k^p:=X_1^p+\dots+X_k^p$, for $k\geq 1$. Then for each $p\in(0,1)$, ${\{S_k^p\}}_k$ is a Bernoulli random walk with parameter $p$. And if $p\geq p'$, then $X_k^p\geq X_k^{p'}$ for all $k$.

Let $M_k^p:=\max\{S_0^p,\dots,S_k^p\}$, and $Z_k^p:=M_k^p-S_k^p$, for $k=0,1\dots$. Observe that for each $p$, the process ${\{Z_k^p\}}_k$ is Markovian. Moreover, it is easy to see that
\begin{equation}
p\geq p' \Rightarrow Z_k^p\leq Z_k^{p'}\ \mbox{for all $k$}.
\label{eq:ordering}
\end{equation}
Finally, and most importantly, Bernoulli random walk satisfies the well-known reflection property
\begin{equation}
(M_n^p-S_n^p,S_n^p)\stackrel{d}{=}(M_n^q,-S_n^q)
\label{eq:reflection-principle}
\end{equation}
for all $n\in\NN$.
(The easiest way to see this is to observe that the time-reversed process $\tilde{S}_k:=S_{n-k}^p-S_{n}^p$, $k=0,1,\dots,n$ is a Bernoulli random walk with parameter $q$, starting at $0$ and ending at $-S_n^p$, with maximum value $M_n^p-S_n^p$.) In particular (reversing the roles of $p$ and $q$),
\begin{equation}
M_n^p\stackrel{d}{=} M_n^q-S_n^q=Z_n^q.
\label{eq:time-reversal}
\end{equation}
It is almost amusing to see how many times this identity must be used in order to prove Theorem \ref{thm:optimal-rules}.

The following lemma holds the key to the proof of Theorem \ref{thm:optimal-rules}.

\begin{lemma} \label{lem:key-inequality}
Let $f:\{0,1,\dots,N\}\to\RR$ be nonincreasing and convex.

(i) If $p\geq 1/2$, then
\begin{equation}
\sE[f(i\vee M_n^p-S_n^p)]\geq \sE\big[f\big(i\vee (M_n^p-S_n^p)\big)\big]
\label{eq:key-inequality}
\end{equation}
for all $n\leq N$ and all $i\geq 0$. 

(ii) If $p>1/2$ and $f$ is strictly decreasing, then strict inequality holds in \eqref{eq:key-inequality} for all $0<n\leq N$ and all $i>0$.

(iii) If $p\geq 1/2$ and $f$ is strictly convex, then strict inequality holds in \eqref{eq:key-inequality} for all $0<n\leq N$ and all $i>0$.
\end{lemma}

\begin{proof}
(i) Let $p\geq 1/2$. We begin by writing
\begin{align*}
\rE&\left[f(i\vee M_n^p-S_n^p)-f\big(i\vee (M_n^p-S_n^p)\big)\right]\\
&=\sum_{l\in\ZZ}\sum_{k\geq l}\left[f(i\vee k-l)-f\big(i\vee (k-l)\big)\right]\sP(M_n^p=k,S_n^p=l)\\
&=\sum_{l>0}\sum_{k\geq l}+\sum_{l<0}\sum_{k\geq 0}=:\Sigma^++\Sigma^-.
\end{align*}
(Note that the terms with $l=0$ vanish.) By \eqref{eq:reflection-principle} and the change of variable $k'=k-l$, $l'=-l$, the second summation becomes
\begin{align*}
\Sigma^-&=\sum_{l<0}\sum_{k\geq 0}\left[f(i\vee k-l)-f\big(i\vee(k-l)\big)\right]\sP(M_n^q=k-l, S_n^q=-l)\\
&=\sum_{l'>0}\sum_{k'\geq l'}\left[f\big(i\vee(k'-l')+l'\big)-f(i\vee k')\right]\sP(M_n^q=k', S_n^q=l').
\end{align*}
The key to further progress is that for $l>0$,
\begin{equation*}
\rP(M_n^p=k, S_n^p=l)\geq \sP(M_n^q=k, S_n^q=l). 
\end{equation*}
(This follows easily by considering the probability of a single path ending at $l$ with maximum $k$.) Since $f$ is nonincreasing and $i\vee k-l\leq i\vee(k-l)$, we have
\begin{equation*}
f(i\vee k-l)-f\big(i\vee(k-l)\big)\geq 0,
\end{equation*}
and therefore,
\begin{equation}
\Sigma^+\geq\sum_{l>0}\sum_{k\geq l}\left[f(i\vee k-l)-f\big(i\vee(k-l)\big)\right]\sP(M_n^q=k, S_n^q=l).
\label{eq:domination}
\end{equation}
Combining these results, we conclude that
\begin{equation}
\Sigma^++\Sigma^-\geq\sum_{l>0}\sum_{k\geq l}\psi(i,k,l)\sP(M_n^q=k, S_n^q=l),
\label{eq:final-inequality}
\end{equation}
where
\begin{align*}
\psi(i,k,l):&=\big[f(i\vee k-l)-f\big(i\vee(k-l)\big)\big]+\big[f\big(i\vee(k-l)+l\big)-f(i\vee k)\big]\\
&=[f(i\vee k-l)-f(i\vee k)]-\big[f\big(i\vee(k-l)\big)-f\big(i\vee(k-l)+l\big)\big].
\end{align*}
Since $i\vee k-l\leq i\vee(k-l)$ and $f$ is convex, it is easy to see that $\psi(i,k,l)\geq 0$. This yields \eqref{eq:key-inequality}.

(ii) Suppose $p>1/2$ and $f$ is strictly decreasing. Let $n>0$ and $i>0$, and put $k=l=n$. Then
\begin{equation*}
f(i\vee k-l)-f\big(i\vee(k-l)\big)=f\big((i-n)^+\big)-f(i)>0.
\end{equation*}
Since $\rP(M_n^p=S_n^p=n)>\sP(M_n^q=S_n^q=n)$, strict inequality holds in \eqref{eq:domination}, and hence in \eqref{eq:key-inequality}.

(iii) Finally, suppose $p\geq 1/2$ and $f$ is strictly convex. Let $n>0$ and $i>0$. Since $i\vee n-n=(i-n)^+<i=i\vee(n-n)$, the strict convexity of $f$ implies that $\psi(i,n,n)>0$. This, together with \eqref{eq:final-inequality} and the obvious fact that $\rP(M_n^q=S_n^q=n)>0$, gives strict inequality in \eqref{eq:key-inequality}. 
\end{proof}

\begin{corollary} \label{cor:consequence}
Let $f$ be as in Lemma \ref{lem:key-inequality}. If $p\geq 1/2$, then
\begin{equation}
\sE[f(i\vee M_n^p-S_n^p)]\geq \sE[f(i\vee M_n^p)]
\label{eq:p-inequality}
\end{equation}
for all $n\leq N$ and all $i\geq 0$. Moreover, if $p>1/2$ and $f$ is strictly decreasing, then strict inequality holds in \eqref{eq:p-inequality} for all $0<n\leq N$ and all $i\geq 0$.
\end{corollary}

\begin{proof}
Let $p\geq 1/2$. Note that in view of \eqref{eq:reflection-principle}, the inequality \eqref{eq:key-inequality} can be stated alternatively as
\begin{equation}
\sE[f(i\vee M_n^p-S_n^p)]\geq \sE[f(i\vee M_n^q)].
\label{eq:alternative-form}
\end{equation}
Since $M_n^q\leq M_n^p$ and $f$ is nonincreasing, we have furthermore
\begin{equation}
\rE[f(i\vee M_n^q)]\geq \sE[f(i\vee M_n^p)].
\label{eq:pq-inequality}
\end{equation}
This, together with \eqref{eq:alternative-form}, gives \eqref{eq:p-inequality}. 

Now suppose $p>1/2$ and $f$ is strictly decreasing. By Lemma \ref{lem:key-inequality}(ii), it suffices to verify strict inequality for $i=0$. But for this value of $i$, \eqref{eq:pq-inequality} holds with strict inequality, since $\rP(M_n^q<M_n^p)>0$ for $n>0$. 
\end{proof}

\begin{proof}[Proof of Theorem \ref{thm:optimal-rules}]
Define the $\sigma$-algebras $\FF_k:=\sigma(\{U_1,\dots,U_k\})$, for $k=0,1,\dots,N$. We prove the slightly stronger statement that, even among stopping rules that can use complete information about the $U_k$'s, the rules given in the statement of the theorem are optimal. Recall that, for a stopping time $\tau$ adapted to $\{\FF_k\}$, the sigma algebra $\FF_\tau$ is defined by the rule
\begin{equation*}
A\in\FF_\tau \Leftrightarrow A\cap\{\tau\leq k\}\in\FF_k\ \mbox{for all $k$}.
\end{equation*}

(i) Consider first the case $p\leq 1/2$. The argument below is adapted from \cite{YYZ}. Let $\tau$ be a stopping time adapted to $\{\FF_k\}$. By conditioning on $\FF_\tau$, we can write
\begin{equation*}
\rE[f(M_N^p-S_\tau^p)]=\sE[G(N-\tau,Z_\tau^p)],
\end{equation*}
where
\begin{equation}
G(k,i):=\sE[f(i\vee M_k^p)].
\label{eq:G}
\end{equation}
Using \eqref{eq:time-reversal} and the stationary and independent increments of the random walk, we obtain similarly
\begin{equation*}
\rE[f(M_N^p)]=\sE[f(Z_N^q)]=\sE[\sE[f(Z_N^q)|\FF_\tau]]
=\sE[D(N-\tau,Z_\tau^q)],
\end{equation*}
where
\begin{equation*}
D(k,i):=\sE[f(i\vee M_k^q-S_k^q)].
\end{equation*}
(See \cite{YYZ}, p.~654 and p.~660 for the details of these calculations in the case $f(k)=e^{-\delta k}$.)
Since $f$ is nonincreasing, $G(k,i)$ is nonincreasing in $i$ for fixed $k$, which by \eqref{eq:ordering} implies that $G(N-\tau,Z_\tau^p)\leq G(N-\tau,Z_\tau^q)$. But by \eqref{eq:alternative-form} with the roles of $p$ and $q$ reversed,
\begin{equation*}
D(k,i)\geq G(k,i)
\end{equation*}
for all $k$ and all $i$. It follows that
\begin{align}
\rE[f(M_N^p-S_\tau^p)]&=\sE[G(N-\tau,Z_\tau^p)]
\leq \sE[G(N-\tau,Z_\tau^q)] \label{eq:G-inequality} \\
&\leq \sE[D(N-\tau,Z_\tau^q)]=\sE[f(M_N^p)],\notag
\end{align}
for any stopping time $\tau$. Thus, the rule $\tau\equiv 0$ is optimal. 

(ii) Assume next that $p\geq 1/2$. Define $G(k,i)$ by \eqref{eq:G}, and let
\begin{equation*}
\tilde{D}(k,i):=\sE[f(i\vee M_k^p-S_k^p)].
\end{equation*}
By Corollary \ref{cor:consequence}, $\tilde{D}(k,i)\geq G(k,i)$, and hence, for any stopping time $\tau$,
\begin{align}
\rE[f(M_N^p-S_\tau^p)]&=\sE[G(N-\tau,Z_\tau^p)]
\leq \sE[\tilde{D}(N-\tau,Z_\tau^p)] \label{eq:GDtilde-inequality}\\
&=\sE[f(Z_N^p)]=\sE[f(M_N^p-S_N^p)]. \notag
\end{align}
Therefore, the rule $\tau\equiv N$ is optimal. 

(iii) Consider finally the case $p=1/2$. Observe that $G(0,i)=D(0,i)=f(i)$ for all $i$, and
$G(k,0)=\sE[f(M_k^p)]=\sE[f(Z_k^q)]=D(k,0)$
for all $k$. Thus, for any stopping time $\tau$ with $S_\tau=M_\tau$ or $\tau=N$ almost surely, 
\begin{equation*}
G(N-\tau,Z_\tau^p)=D(N-\tau,Z_\tau^p)=D(N-\tau,Z_\tau^q)
\end{equation*}
(since $p=q$), and hence, for any such $\tau$,
\begin{equation}
\rE[f(M_N^p-S_\tau^p)]=\sE[f(M_N^p)]=\sup_{\tau'}\sE[f(M_N^p-S_{\tau'})], 
\label{eq:tau-is-optimal}
\end{equation}
where the last equality follows by part (i). 
\end{proof}

\begin{proof}[Proof of Theorem \ref{thm:unique}]
(i) Let $p<1/2$, and suppose $f$ is not constant. Since $f$ is nonincreasing and convex, this implies that $f(0)>f(i)$ for all $i>0$. It follows that $G(k,0)>G(k,i)$ for all $i>0$ and all $k$, since obviously $f(M_k^p)\geq f(i\vee M_k^p)$, and
\begin{equation*}
\rP\left[f(M_k^p)>f(i\vee M_k^p)\right]\geq \sP(M_k^p=0)>0.
\end{equation*}
Now consider a stopping time $\tau$ with $\tau>0$. Then
\begin{align*}
\rP(Z_\tau^q=0,Z_\tau^p>0)&\geq \sP(Z_k^q=0\ \mbox{and $Z_k^p>0$ for $k=1,\dots,N$})\\
&\geq \sP(X_k^q=1 \ \mbox{and $X_k^p=-1$ for $k=1,\dots,N$})\\
&=(q-p)^N>0.
\end{align*}
(Note that this holds for {\em any} random time $\tau$, not just for stopping times.) It therefore follows that $\rE[G(N-\tau,Z_\tau^q)]>\sE[G(N-\tau,Z_\tau^p)]$, which is strict inequality in \eqref{eq:G-inequality}.

(ii) Suppose next that $p>1/2$ and $f$ is strictly decreasing. Then strict inequality holds in Corollary \ref{cor:consequence} for $n>0$ and all $i$. But this yields strict inequality in \eqref{eq:GDtilde-inequality} for any stopping time $\tau$ with $\rP(\tau<N)>0$.

(iii) Finally, assume $p=1/2$, and let $f$ be strictly convex. If $N=1$, the only stopping times are $\tau\equiv 0$ and $\tau\equiv 1$, which both satisfy the condition in Theorem \ref{thm:optimal-rules}(iii). So assume $N\geq 2$. By Lemma \ref{lem:key-inequality}(iii), strict inequality holds in \eqref{eq:key-inequality} for all $i>0$. Thus, if $\tau$ is a stopping time with the property that $\rP(M_\tau^p-S_\tau^p>0\ \mbox{and}\ \tau<N)>0$, then
\begin{equation*}
\rE[D(N-\tau,Z_\tau^p)]>\sE[G(N-\tau,Z_\tau^p)], 
\end{equation*}
and so the first equality in \eqref{eq:tau-is-optimal} is replaced with ``$<$". 
\end{proof}

\section{The maximum of Brownian motion} \label{sec:Brownian}

The key to the proof of Theorem \ref{thm:BM-optimal-rules} is the following analog of Lemma \ref{lem:key-inequality}. It makes use of the well-known fact, analogous to \eqref{eq:reflection-principle}, that
\begin{equation}
(M_t^\lambda-B_t^\lambda,B_t^\lambda)\stackrel{d}{=}(M_t^{-\lambda},-B_t^{-\lambda}).
\label{eq:BM-reflection-principle}
\end{equation}

\begin{lemma} \label{lem:BM-key}
Let $f:[0,\infty)\to\RR$ be nonincreasing and convex.

(i) If $\lambda\geq 0$, then
\begin{equation}
\rE\left[f(x\vee M_t^\lambda-B_t^\lambda)\right]\geq
\sE\left[f\big(x\vee (M_t^\lambda-B_t^\lambda)\big)\right]
\label{eq:BM-key-inequality}
\end{equation}
for all $t\geq 0$ and all $x\geq 0$.

(ii) If $\lambda>0$ and $f$ is not constant, then strict inequality holds in \eqref{eq:BM-key-inequality} for all $t>0$ and all $x>0$.

(iii) If $\lambda=0$ and $f$ is not linear, then strict inequality holds in \eqref{eq:BM-key-inequality} for all $t>0$ and all $x>0$.
\end{lemma}

\begin{proof}
(i) The inequality is trivial when $t=0$, so assume $t>0$.
Let $h(s,b;\lambda)$ be the joint density function of $(M_t^\lambda,B_t^\lambda)$.
Note that in view of \eqref{eq:BM-reflection-principle}, or by \eqref{eq:joint-density} below,
\begin{equation}
h(s,b;\lambda)=h(s-b,-b;-\lambda).
\label{eq:density-relationship}
\end{equation}
As in the proof of Lemma \ref{lem:key-inequality}, we begin by writing
\begin{align*}
\rE&\left[f(x\vee M_t^\lambda-B_t^\lambda)-f\big(x\vee(M_t^\lambda-B_t^\lambda)\big)\right]\\
&=\int_{b\in\RR}\int_{s>b}
\big[f(x\vee s-b)-f\big(x\vee(s-b)\big)\big]h(s,b;\lambda)\,ds\,db\\
&=\int_{b>0}\int_{s>b}+\int_{b<0}\int_{s>0}=:I^++I^-.
\end{align*}
Using \eqref{eq:density-relationship} and the change of variable $z=s-b,\ b'=-b$, we can write $I^-$ as
\begin{align*}
I^-&=\int_{b'>0}\int_{z>b'}
\big[f\big(x\vee(z-b')+b'\big)-f(x\vee z)\big]h(z,b';-\lambda)\,dz\,db'\\
&=\int_{b>0}\int_{s>b}\big[f\big(x\vee(s-b)+b\big)-f(x\vee s)\big]h(s,b;-\lambda)\,ds\,db,
\end{align*}
where the last equality follows simply by renaming the variables. Recall (see, e.g., equation (3.2) of \cite{DuToit}) that for fixed $t$, $h(s,b;\lambda)$ is given by the formula
\begin{equation}
h(s,b;\lambda)=\sqrt{\frac{2}{\pi}}\frac{2s-b}{t^{3/2}}e^{-(2s-b)^2/2t}e^{\lambda(b-\lambda t/2)}
\label{eq:joint-density}
\end{equation}
for all $s\geq 0$ and $b\leq s$. It follows that for all $b>0$ and $s\geq b$,
\begin{equation*}
h(s,b;\lambda)\geq h(s,b;-\lambda),
\end{equation*}
with strict inequality if $\lambda>0$. (Note that there does not seem to be a direct probabilistic argument for this last inequality; instead, we must rely on the specific form of the density formula \eqref{eq:joint-density}.)
Since $f$ is nonincreasing and $x\vee s-b\leq x\vee(s-b)$ for $b>0$, we have
\begin{equation*}
f(x\vee s-b)-f\big(x\vee(s-b)\big)\geq 0, \qquad \mbox{for $b>0$}.
\end{equation*}
Thus,
\begin{equation}
I^+\geq\int_{b>0}\int_{s>b}
\big[f(x\vee s-b)-f\big(x\vee(s-b)\big)\big]h(s,b;-\lambda)\,ds\,db.
\label{eq:positive-part}
\end{equation}
Putting these results together, we conclude that
\begin{equation}
I^++I^-\geq\int_{b>0}\int_{s>b}\psi(x,s,b)h(s,b;-\lambda)\,ds\,db,
\label{eq:BM-together}
\end{equation}
where
\begin{equation*}
\psi(x,s,b):=f(x\vee s-b)-f\big(x\vee(s-b)\big)+f\big(x\vee(s-b)+b\big)-f(x\vee s).
\end{equation*}
As in the proof of Lemma \ref{lem:key-inequality}, the convexity of $f$ implies $\psi(x,s,b)\geq 0$. Thus, the proof of \eqref{eq:BM-key-inequality} is complete.

(ii) Suppose now that $\lambda>0$ and $f$ is not constant. Fix $x>0$. Since $f$ is nonincreasing and convex, we can choose $\delta>0$ so small that $2\delta<x$, and $f(2\delta)>f(x)$. But then, on the small square $x-\delta<b<x<s<x+\delta$, we have
\begin{equation*}
f(x\vee s-b)-f\big(x\vee(s-b)\big)=f(s-b)-f(x)\geq f(2\delta)-f(x)>0.
\end{equation*}
Since $h(s,b;\lambda)>h(s,b;-\lambda)$ on this small square, strict inequality results in \eqref{eq:positive-part}, and hence in \eqref{eq:BM-key-inequality}.

(iii) Suppose finally that $\lambda=0$ and $f$ is not linear. Then there exists a point $x_0>0$ such that for all $x>x_0$ and all $u>0$, $f(0)-f(u)>f(x)-f(x+u)$. Choose $n\in\NN$ such that $nx>x_0$. Then for $s=b=nx$, $\psi(x,s,b)=f(0)-f(x)+f\big((n+1)x\big)-f(nx)>0$. By continuity of $\psi$, it follows that $\psi>0$ on a small square of positive $h(s,b;-\lambda)$-density. Putting this back in \eqref{eq:BM-together} gives strict inequality in \eqref{eq:BM-key-inequality}. 
\end{proof}

\begin{corollary} \label{cor:BM-consequence}
Let $f$ be as in Lemma \ref{lem:BM-key}. If $\lambda\geq 0$, then
\begin{equation}
\sE[f(x\vee M_t^\lambda-B_t^\lambda)]\geq \sE[f(x\vee M_t^\lambda)]
\label{eq:lambda-inequality}
\end{equation}
for all $t\geq 0$ and all $x\geq 0$. Moreover, if $\lambda>0$ and $f$ is not constant, then strict inequality holds in \eqref{eq:lambda-inequality} for all $t>0$ and all $x\geq 0$.
\end{corollary}

\begin{proof}
Let $\lambda\geq 0$. In view of \eqref{eq:BM-reflection-principle}, the inequality \eqref{eq:BM-key-inequality} is equivalent to
\begin{equation}
\sE[f(x\vee M_t^\lambda-B_t^\lambda)]\geq \sE[f(x\vee M_t^{-\lambda})].
\label{eq:BM-alternative-form}
\end{equation}
(Note that \eqref{eq:BM-alternative-form} generalizes the key inequality (4.28) in \cite{DuToit}.)
Since $M_t^{-\lambda}\leq M_t^\lambda$ and $f$ is nonincreasing, we have
\begin{equation}
\rE[f(x\vee M_t^{-\lambda})]\geq \sE[f(x\vee M_t^\lambda)].
\label{eq:pm-lambda-inequality}
\end{equation}
This, together with \eqref{eq:BM-alternative-form}, gives \eqref{eq:lambda-inequality}. 

Now suppose that $\lambda>0$ and $f$ is not constant. By Lemma \ref{lem:BM-key}(ii), it suffices to verify strict inequality for $x=0$.
Since $f$ is nonincreasing and convex, there exists $x_0>0$ such that $f$ is strictly decreasing on $[0,x_0]$. Clearly, $\rP(M_t^{-\lambda}<M_t^\lambda<x_0)>0$ for $t>0$. As a result, strict inequality holds in \eqref{eq:pm-lambda-inequality} for $x=0$.
\end{proof}

\begin{proof}[Proof of Theorem \ref{thm:BM-optimal-rules}]
a) {\em Optimality.} We first prove that the rules given in the statement of the theorem are optimal. Let
\begin{equation*}
Z_t^\lambda:=M_t^\lambda-B_t^\lambda, \qquad t\geq 0,
\end{equation*}
and note that for fixed $t$, $Z_t^\lambda$ is pointwise nonincreasing in $\lambda$.

(i) Assume first that $\lambda\leq 0$. Define the functions
\begin{equation*}
G(t,x):=\sE[f(x\vee M_t^\lambda)], \qquad D(t,x):=\sE[f(x\vee M_t^{-\lambda}-B_t^{-\lambda})].
\end{equation*}
Let $\tau\leq T$ be any stopping time adapted to the filtration $(\FF_t^B)$. As in the proof of Theorem \ref{thm:optimal-rules}, we have
\begin{equation*}
\rE[f(M_T^\lambda-B_\tau^\lambda)]=\sE[G(T-\tau,Z_\tau^\lambda)].
\end{equation*}
Using \eqref{eq:BM-reflection-principle}, the stationary and independent increments of Brownian motion and the strong Markov property of the process $(Z_t)$, we obtain
\begin{equation*}
\rE[f(M_T^\lambda)]=\sE[f(Z_T^{-\lambda})]=\sE[D(T-\tau,Z_\tau^{-\lambda})].
\end{equation*}
(For the details of these calculations, see \cite{DuToit}, p.~987 and p.~1004.)
Since $f$ is nonincreasing, $G(t,x)$ is nonincreasing in $x$ for fixed $t$. It follows that $G(T-\tau,Z_\tau^\lambda)\leq G(T-\tau,Z_\tau^{-\lambda})$. Furthermore, \eqref{eq:BM-alternative-form} with $\lambda$ replaced by $-\lambda$ gives $D(t,x)\geq G(t,x)$, for all $t$ and all $x$. As a result,
\begin{align}
\rE[f(M_T^\lambda-B_\tau^\lambda)]&=\sE[G(T-\tau,Z_\tau^\lambda)]
\leq \sE[G(T-\tau,Z_\tau^{-\lambda})] \label{eq:BM-GG-inequality}\\
&\leq \sE[D(T-\tau,Z_\tau^{-\lambda})]=\sE[f(M_T^\lambda)].\label{eq:BM-GD-inequality}
\end{align}
Since this holds for any stopping time $\tau$, it follows that the rule $\tau\equiv 0$ is optimal.

(ii) Consider next the case $\lambda\geq 0$. Let
\begin{equation*}
\tilde{D}(t,x):=\sE[f(x\vee M_t^\lambda-B_t^\lambda)].
\end{equation*}
Then Corollary \ref{cor:BM-consequence} implies that $\tilde{D}(t,x)\geq G(t,x)$, and hence,
\begin{align}
\sE[f(M_T^\lambda-B_\tau^\lambda)]&=\sE[G(T-\tau,Z_\tau^\lambda)]
\leq \sE[\tilde{D}(T-\tau,Z_\tau^{\lambda})]
\label{eq:BM-GDtilde-inequality}\\
&=\sE[f(Z_T^\lambda)]=\sE[f(M_T^\lambda-B_T^\lambda)]\notag
\end{align}
for any stopping time $\tau$. Thus, the rule $\tau\equiv T$ is optimal.

(iii) Suppose finally that $\lambda=0$. Then $G(0,x)=D(0,x)$ for all $x$, and $G(t,0)=D(t,0)$ for all $t$. Thus, for any stopping time $\tau$ satisfying \eqref{eq:indifference},
\begin{equation*}
G(T-\tau,Z_\tau^\lambda)=D(T-\tau,Z_\tau^\lambda)
=D(T-\tau,Z_\tau^{-\lambda}),
\end{equation*}
so that (see \eqref{eq:BM-GG-inequality} and \eqref{eq:BM-GD-inequality})
\begin{equation}
\sE[f(M_T^\lambda-B_\tau^\lambda)]=\sE[f(M_T^\lambda)].
\label{eq:BM-tau-is-optimal}
\end{equation}
By part (i) of the theorem, this implies that $\tau$ is optimal.

\bigskip
b) {\em Uniqueness.} We next verify the uniqueness claims in Theorem \ref{thm:BM-optimal-rules}.

(i) Assume first that $\lambda<0$. While Lemma \ref{lem:BM-key} provides strict inequality in \eqref{eq:BM-GD-inequality} for the majority of stopping times, it does not do so for stopping times $\tau$ of the form \eqref{eq:indifference}. Therefore, we establish strict inequality in \eqref{eq:BM-GG-inequality} instead. First, since $f$ is non-constant, nonincreasing and convex, there exists a point $x_0>0$ such that $f$ is strictly decreasing on $[0,x_0]$. It is easy to see that the same is then true for $G(t,\cdot)$ for any fixed $t$, including $t=0$. Let $\tau\leq T$ be a stopping time with $\rP(\tau>0)>0$. We show first that
\begin{equation}
\rP(0<Z_\tau<x_0)>0,
\label{eq:non-absorption}
\end{equation}
where we write $Z_t$ for $Z_t^{\lambda}$. Choose $t_0>0$ so that $\rP(\tau>t_0)>0$, and let
\begin{equation*}
\tau_0:=\min\{t_0,\tau(x_0/2)\},
\end{equation*}
where $\tau(x):=\inf\{t>0:Z_t\geq x\}$ for $x>0$. Then $\tau_0$ is a stopping time adapted to $(\FF_t^B)$, and so $\{\tau>\tau_0\}\in\FF_{\tau_0}^B$. Moreover, $\rP(\tau>\tau_0)\geq\sP(\tau>t_0)>0$, and $\rP(Z_{\tau_0}>0)=\sP(Z_{t_0}>0)=1$. Thus, the set $\{\tau>\tau_0, Z_{\tau_0}>0\}$ lies in $\FF_{\tau_0}^B$ and has positive probability. On this set,
\begin{equation*}
\rP\left(0<Z_\tau<x_0\big|\,\FF_{\tau_0}^B\right)\geq \sP\left(0<Z_t<x_0\ \mbox{for}\ \tau_0\leq t\leq T\big|\,\FF_{\tau_0}^B\right)>0,
\end{equation*}
by the strong Markov property of $(Z_t)$ and the fact that $(Z_t)$ behaves like Brownian motion with drift as long as it does not hit $0$. But then
\begin{equation*}
\rP(0<Z_\tau<x_0)=\sE\left[\rP\left(0<Z_\tau<x_0\big|\,\FF_{\tau_0}^B\right)\right]>0,
\end{equation*}
proving \eqref{eq:non-absorption}.

Next, a moment of reflection shows that $Z_t^\lambda=Z_t^{-\lambda}$ if and only if $Z_t^{\lambda}=0$. Thus, by \eqref{eq:non-absorption},
\begin{equation*}
\rP(Z_\tau^{-\lambda}<Z_\tau^{\lambda}<x_0)=\sP(0<Z_\tau^{\lambda}<x_0)>0.
\end{equation*}
Along with the fact that $G(t,\cdot)$ is strictly decreasing on $[0,x_0]$ for all $t\geq 0$, this yields strict inequality in \eqref{eq:BM-GG-inequality}.

(ii) Consider next the case $\lambda>0$. Then strict inequality holds in Corollary~\ref{cor:BM-consequence} for $t>0$ and all $x$. But this yields strict inequality in \eqref{eq:BM-GDtilde-inequality} above for any stopping time $\tau$ with $\rP(\tau<T)>0$.

(iii) Assume finally that $\lambda=0$, and $f$ is not linear. By Lemma \ref{lem:BM-key}(iii), strict inequality holds in \eqref{eq:BM-key-inequality} for all $t>0$ and all $x>0$. Thus, for any stopping rule $\tau$ such that $\rP(M_\tau^\lambda-B_\tau^\lambda>0\ \mbox{and}\ \tau<T)>0$,
\begin{equation*}
\rE[D(T-\tau,Z_\tau^\lambda)]>\sE[G(T-\tau,Z_\tau^\lambda)], 
\end{equation*}
and so the equality in \eqref{eq:BM-tau-is-optimal} is replaced with ``$<$". 
\end{proof}

\section*{Acknowledgements}
This work was prepared while the author was on sabbatical in Kyoto, Japan. The author wishes to thank the Kyoto University Mathematics Department and the Research Institute for Mathematical Sciences (RIMS) for their warm hospitality during 2009.

\end{document}